\documentclass[11pt,twoside]{amsart}
\usepackage[all]{xy}   
\usepackage {amssymb,latexsym,amsthm,amsmath}\usepackage{amsmath}
\usepackage{xcolor}
\usepackage{amsthm}
\usepackage{calrsfs}
\usepackage{mathtools}
\usepackage{tikz-cd}

\long\def\symbolfootnote[#1]#2{\begingroup%
\def\thefootnote{\fnsymbol{footnote}}\footnote[#1]{#2}\endgroup}

\makeatletter
\def\imod#1{\allowbreak\mkern10mu({\operator@font mod}\,\,#1)}
\makeatother

\newtheorem{theorem}{Theorem}[section]
\newtheorem{lemma}[theorem]{Lemma}
\newtheorem{corollary}[theorem]{Corollary}
\newtheorem{proposition}[theorem]{Proposition}
\newtheorem{definition}[theorem]{Definition}
\newtheorem*{theorem*}{Theorem}
\theoremstyle{definition}
\newtheorem{remark}[theorem]{Remark}
\newtheorem{example}[theorem]{Example}
\newtheorem{question}[theorem]{Question}
\numberwithin{equation}{section}

\newcommand{\ignore}[1]{}

\newcommand{\mynote}[1]{}
\begin{document}
\setcounter{section}{0}
% document information
\title{Commuting probability in algebraic groups}
\author{Shripad M. Garge.}
\address{Department of Mathematics, Indian Institute of Technology, Powai, Mumbai, 400 076 India.}
\email{shripad@math.iitb.ac.in}
\email{smgarge@gmail.com}
%\subjclass[2010]{20G15, 11E72}

%%%%%%%%%%%%%%%%%%%%%%%%%%%%%%%%%%%%

\begin{abstract}
We introduce the notion of commuting probability, $p(G)$, for an algebraic group $G$. 
This notion is inspired by the corresponding notions in finite groups and compact groups. 
The computation of $p(G)$ for reductive groups is readily done using the notion of $z$-classes. 
We introduce two generalisations of this relation, $iz$-equivalence and $dz$-equivalence.
These notions lead us naturally to the notion of a regular element in $G$.
Finally, with the help of this notion of regular elements, we compute $p(G)$ for a connected, linear algebraic group $G$. 
We also compute the set of limit points of the numbers $p(G)$ as $G$ varies over the classes of reductive groups, solvable groups and nilpotent groups.
\end{abstract}

\maketitle
%%%%%%%%%%%%%%%%%%%%%%%%%%%%%%%%%%%%
\section{Introduction}
\emph{What is the probability that two randomly chosen elements of a finite group commute?}

\vskip2mm
This was the question asked by Erd\"{o}s and Turan in \cite{ET} where they formulated the notion of commuting probability for a finite group. 
If $C$ denotes the set of ordered pairs of commuting elements of a finite $G$ then $p(G)$, the commuting probability of $G$, is defined to be the ratio $|C|/|G \times G|$. 
It is proved in the same paper that this number is equal to $k/|G|$ where $k$ is the number of conjugacy classes in $G$. 
Almost immediately, Gustafson (\cite{G}), generalised the notion of commuting probability to compact groups where he used the Haar measure in the definition. 
There has been a flurry of activity around this topic for many years since then. 

The purpose of the present paper is to introduce the notion of commuting probability in the case of algebraic groups. 
In algebraic geometry, the natural measure of a certain algebraic set is its dimension. 
If $G$ is a linear algebraic group then we define the set $C(G)$ to be the set of ordered pairs of commuting elements in $G$, $C(G) := \{(a, b) \in G \times G: ab = ba\}$. 
This is a Zariski closed subset of $G \times G$. 
We define the commuting probability of $G$, denoted by $p(G)$, as follows
$$p(G) := \frac{\dim(C(G))}{\dim(G \times G)} = \frac{\dim(C(G))}{2\dim(G)} .$$
For a finite group $G$, which can be considered as an algebraic group with $\dim(G) = 0$, we define $p(G) = 1$. 

We note a basic lemma which computes the commuting probability of a direct product. 

\begin{lemma}\label{direct}
	If $G = H_1 \times H_2$ then 
	$$p(G) = \dfrac{\dim(H_1)p(H_1) + \dim(H_2)p(H_2)}{\dim (H_1) + \dim(H_2)} .$$
\end{lemma}

\begin{proof}
	Since $G = H_1 \times H_2$, it follows that $C(G)$ is isomorphic to $C(H_1) \times C(H_2)$ and therefore $\dim(C(G)) = \dim(C(H_1)) + \dim(C(H_2))$ which further equals $2\dim(H_1)p(H_1) + 2\dim(H_2)p(H_2)$.
	As $\dim(G) = \dim(H_1) + \dim(H_2)$, the lemma now follows.
\end{proof}

If $G$ is a connected, linear algebraic group then so is $G \times G$ and then the dimension of every proper closed subset of $G \times G$ is smaller than $\dim(G \times G)$. 
Hence for a connected $G$, $p(G) = 1$ if and only if $G$ is abelian. 

If $G$ is not connected, then $p(G)$ can be 1 without $G$ being abelian. 
Consider $G = G_1 \times G_2$ where $G_1$ is abelian and $G_2$ is a finite non-abelian group. 
Then by the above lemma, $p(G) = 1$ but $G$ is not abelian. 

To avoid such cases, we will work with connected groups in the remainder of this paper. 
In any case, $p(G) = p(G^0)$, so we are not losing out on the generality.
We also choose to work over $\mathbb{C}$, though the results remain valid over an algebraically closed field, indeed over a perfect field.

We begin the study of $p(G)$ with the case of reductive groups in Section 2. 
If $G$ is a reductive group of dimension $n$ and rank $r$ then we prove in Theorem \ref{reductive} that $p(G) = (n+r)/2n$.
We also list the numbers $p(G)$ for all simple groups. 
These calculations are used later while computing the limit points of the set of the numbers $p(G)$. 

The main tool used in the proof of Theorem \ref{reductive} is the notion of $z$-classes which is not useful in the case of a non-reductive connected $G$. 
Hence we introduce the notions of $iz$-equivalence and $dz$-equivalence, generalising the $z$-equivalence in the third section. 
We study the corresponding equivalence classes and make some interesting observations, which may be of independent interest.
These notions lead us to the notion of a regular element in $G$.

The dimension of the centralizer of a regular element in $G$ is called the regular rank of $G$. 
We prove in Theorem \ref{general} that if $G$ is a connected, linear algebraic group with dimension $n$ and regular rank $r$ then $p(G) = (n+r)/2n$.
In this section, Section 4, we also list the numbers $p(U)$ where $U$ is the unipotent radical of a fixed Borel subgroup in a simple algebraic group. 

By the Theorem \ref{general}, $p(G)$ is a rational number $> 1/2$. 
In Section 5, we compute all rational numbers that can occur as $p(G)$ for various classes of $G$, namely the reductive groups, solvable groups and nilpotent groups. 
As a consequence, we compute the limit points of these numbers for all these above classes. 
The set of limit points in all these cases is equal to the interval $[1/2, 1]$. 

We then compare the results on $p(G)$ for algebraic groups with the corresponding results in finite groups. 
We point out some similarities and many differences between these two notions. 
This occupies the Section 6. 
In the seventh section, we point out the compatibilities of $p(G)$ for a semisimple group $G$ defined over $\mathbb{F}_q$ with the corresponding finite groups of Lie type, in the asymptotic sense. 
We close the paper with some concluding remarks in the last section.
We also indicate some questions that may of interest to the community.

%%%%%%%%%%%%%%%%%%%%%%%%%%%%%%%%%%%%

\section*{Acknowledgements}
It is a pleasure to thank Dipendra Prasad, Saurav Bhaumik, Anuradha Garge and Anupam Kumar Singh for many useful discussions. 
Anupam, in particular, has been a catalyst for this paper in more than one ways. 
%It is a pleasure to dedicate this work to 20 years of friendship with him!

%%%%%%%%%%%%%%%%%%%%%%%%%%%%%%%%%%%%
\section{Reductive groups}
Let $G$ be a complex reductive group. 
In this section, we compute the commuting probability $p(G)$ of $G$ in terms of the dimension and rank of $G$. 
The rank of a reductive group is the dimension of a maximal torus in $G$. 
Since we compute dimensions, it will be sufficient to work with the group $G(\mathbb{C})$. 
In the computation, we will use the notion of a $z$-class. 
Let us first recall it. 

Two elements $g, h \in G(\mathbb{C})$ are said to be \emph{$z$-equivalent} if their centralisers, $Z_{G(\mathbb{C})}(g)$ and $Z_{G(\mathbb{C})}(h)$, are conjugate within $G(\mathbb{C})$. 
This is an equivalence relation and the corresponding equivalence classes are called \emph{$z$-classes}. 
This equivalence is weaker than the one given by the conjugacy relation. 
Each $z$-class is a union of certain conjugacy classes.

It is known that if $G$ is a reductive group defined over $\mathbb{C}$ then the number of $z$-classes in $G(\mathbb{C})$ is finite. 
The result is true in a more general setting and we refer the reader to \cite{GS} for a detailed discussion. 
For the present paper, the finiteness result over $\mathbb{C}$ is sufficient. 

\begin{theorem}\label{reductive}
Let $G$ be a complex reductive group with dimension $n$ and rank $r$. 
The commuting probability of $G$, $p(G)$, is equal to
$$\frac{n + r}{2n} = \frac{1}{2} + \frac{r}{2n}.$$
\end{theorem}

\begin{proof}
We will be working throughout the proof with the set of $\mathbb{C}$-rational points without mentioning the field $\mathbb{C}$ explicitly. 
So, whenever we talk about a subset of an algebraic set $X$, we will always mean a subset of $X(\mathbb{C})$. 

Let $\mathcal{C}$ denote the set of conjugacy classes in $G$ and let $\mathfrak{c}$ be a conjugacy class in $G$. 
The set $C(G)$ is equal to the set
$$\cup_{g \in G} \{g\} \times Z_G(g) = \bigcup_{\mathfrak{c} \in \mathcal{C}} \bigg(\cup_{g \in \mathfrak{c}} \big(\{g\} \times Z_G(g)\big)\bigg) = \bigcup_{\mathfrak{c} \in \mathcal{C}} C(G)_{\mathfrak{c}} .$$
If $g, h \in \mathfrak{c}$ then $Z_G(g)$ and $Z_G(h)$ are conjugate in $G$, in particular their dimensions are the same hence $\dim C(G)_{\mathfrak{c}}$ is equal to $\dim(G)$ for each $\mathfrak{c} \in \mathcal{C}$. 

Let $\mathcal{Z}$ denote the set of $z$-classes in $G$. 
This is a finite set and 
$$\mathcal{C} = \cup_{\mathfrak{z} \in \mathcal{Z}} \big(\mathfrak{z}/\mathfrak{c_{\mathfrak{z}}}\big)$$
where the set $\mathfrak{z}/\mathfrak{c_{\mathfrak{z}}}$ is the set of conjugacy classes in $\mathfrak{z}$. 

Since $\mathcal{Z}$ is a finite set, there is one $z$-class in $G$ which is dense in $G$. 
Indeed, the set of regular semisimple elements forms a $z$-class in $G$ which is dense in $G$. 
We denote this $z$-class by $\mathfrak{z}_0$. 
It then follows that $\bigcup_{\mathfrak{c} \subseteq \mathfrak{z}_0} C(G)_{\mathfrak{c}}$ is dense in $C(G)$ and hence $\dim(C(G))$ is equal to the dimension of $\bigcup_{\mathfrak{c} \subseteq \mathfrak{z}_0} C(G)_{\mathfrak{c}}$.

Let us fix a conjugacy class $\mathfrak{c}$ in $\mathfrak{z}_0$. 
Then the dimension of $\bigcup_{\mathfrak{c} \subseteq \mathfrak{z}_0} C(G)_{\mathfrak{c}}$ is equal to $\dim(\mathfrak{z}_0/\mathfrak{c}) + \dim(G)$.
Further, $\dim(\mathfrak{z}_0/\mathfrak{c})$ is $\dim(G) - \dim(\mathfrak{c})$ which is $\dim T$ for a maximal torus in $G$.  
Hence $\dim(C(G)) = \dim (G) + \dim (T) = n + r$ and 
$$p(G) = \frac{n + r}{2n} = \frac{1}{2} + \frac{r}{2n}= \frac{1}{2} + \frac{\mathrm{rank}(G)}{2\dim(G)}.$$
\end{proof}

\begin{remark}
The analysis in the above proof implies that the dimension of the variety of conjugacy classes in a reductive $G$ is equal to the dimension of a maximal torus in $G$. 
This is a well-known fact, see \cite[6.4]{St} for instance. 
However, we have used $z$-classes because we use a generalisation of this notion in the general case.
The reason for that being that we have no information about the dimension of the variety of conjugacy classes in $G$ in the general case. 
We do not even know if the set of conjugacy classes forms a (quasi-projective) variety. 
\end{remark}

\begin{remark}\label{simple}
Using the above theorem, we note commuting probabilities of some groups, including all simple algebraic groups.
\begin{enumerate}
\item $p(GL_n) = \dfrac{1}{2} + \dfrac{n}{2n^2} = \dfrac{1}{2} + \dfrac{1}{2n} = \dfrac{n+1}{2n}$, for $n \geq 1$.
\vskip1mm
\item $p(SL_n) = \dfrac{1}{2} + \dfrac{n-1}{2(n^2-1)} = \dfrac{1}{2} + \dfrac{1}{2(n + 1)}$, for $n \geq 1$.
\vskip1mm
\item $p(SO_{2n+1}) = \dfrac{1}{2} + \dfrac{n}{2(2n^2 + n)} = \dfrac{1}{2} + \dfrac{1}{2(2n + 1)}$, for $n \geq 2$.
\vskip1mm
\item $p(Sp_{2n}) = \dfrac{1}{2} + \dfrac{n}{2(2n^2 + n)} = \dfrac{1}{2} + \dfrac{1}{2(2n + 1)}$, for $n \geq 3$.
\vskip1mm
\item $p(SO_{2n}) = \dfrac{1}{2} + \dfrac{n}{2(2n^2 - n)} = \dfrac{1}{2} + \dfrac{1}{2(2n - 1)}$, for $n \geq 4$.
\vskip1mm
\item $p(G_2) = \dfrac{4}{7}$, \hskip2mm $p(F_4) = p(E_6) = \dfrac{7}{13}$, \hskip2mm $p(E_7) = \dfrac{10}{19}$ and $p(E_8) = \dfrac{16}{31}$.
\end{enumerate}
\end{remark}

%%%%%%%%%%%%%%%%%%%%%%%%%%%%%%%%%%%%
\section{Regular elements and regular rank}

We notice that the following three properties of $z$-classes have been crucial in the proof of Theorem \ref{reductive}. 
\begin{enumerate}
\item A $z$-class is a union of conjugacy classes, 
\item a reductive $G$ contains a dense $z$-class, and,
\item the centralisers of two elements in a $z$-class (are conjugate within $G$, hence are isomorphic as abstract groups, and hence they) have the same dimensions.
\end{enumerate}

\begin{example}
A non-reductive $G$ need not contain a dense $z$-class.
Let $U$ denote the group of $3 \times 3$ upper triangular matrices over $\mathbb{C}$.
Then for
$$g = \begin{pmatrix}
1 & a & b \\ & 1 & c \\ & & 1
\end{pmatrix}, 
\hskip5mm
Z_U(g) = \left\{\begin{pmatrix}
1 & x & y \\ & 1 & z \\ & & 1
\end{pmatrix}: cx = az \right\} .$$
The group $U/Z(U)$ is abelian. 
If two centralisers in $U$ are conjugate then their images in $U/Z(U)$ must be the same. 
Thus, the $z$-classes in $U$ correspond to the ratios $\frac{a}{c} \in (\mathbb{C} \cup {\infty})$ along with the central $z$-class. 

The dimension of each $z$-class is less than $3$. 
In particular, there is no dense $z$-class in $U$. 
\end{example}

Taking a cue from the above three properties we make the following two definitions generalising the notion of $z$-equivalence. 

\begin{definition}
\noindent $(1)$ Two elements $x$ and $y$ in a group $G$ are called \emph{$iz$-equivalent} if the centralisers, $Z_G(x)$ and $Z_G(y)$, are isomorphic. 

\noindent $(2)$ Two elements $x$ and $y$ in an algebraic group $G$ are called \emph{$dz$-equivalent} if the centralisers, $Z_G(x)$ and $Z_G(y)$, have the same dimension. 
\end{definition}

The $iz$-equivalence is also introduced by Dilpreet Kaur and Uday Bhaskar Sharma.
We are borrowing the name, $iz$-equivalence, with their kind permission.

The above two relations are indeed equivalence relations. 
The corresponding equivalence classes will be called $iz$-classes and $dz$-classes, respectively. 
We have the following hierarchy in terms of the strengths of these equivalences:

\begin{itemize}
\item a $dz$-class is a union of $iz$-classes, 
\item an $iz$-class is a union of $z$-classes and 
\item a $z$-class is a union of conjugacy classes.
\end{itemize}

These equivalence relations are, in general, different as the following examples show.

\begin{example}
The $z$-class of central elements in a group is not always a single conjugacy class. 
\end{example}

\begin{example}\label{iz-but-not-z}
We once again consider the group $U$ of $3 \times 3$ upper triangular matrices over $\mathbb{C}$.
If we take 
$$x = \begin{pmatrix}
1 & 1 & 1 \\ & 1 & 0 \\ & & 1
\end{pmatrix},
y = \begin{pmatrix}
1 & 0 & 1 \\ & 1 & 1 \\ & & 1
\end{pmatrix}
\hskip5mm
\mathrm{then}
\hskip5mm
Z_U(x) = \begin{pmatrix}
1 & a & b \\ & 1 & 0 \\ & & 1
\end{pmatrix},
Z_U(y) = \begin{pmatrix}
1 & 0 & c \\ & 1 & d \\ & & 1
\end{pmatrix}$$
as $a, b, c, d \in \mathbb{G}_a$.
These centralisers are isomorphic to $\mathbb{G}_a^2$. 
But they are not conjugate as their images in the quotient $U/Z(U)$, which is abelian, are different one dimensional subgroups of $U/Z(U)$. 

Thus, $x, y$ are $iz$-equivalent but not $z$-equivalent. 
\end{example}

\begin{example}\label{dz-but-not-iz}
Finally, we take $x$ to be a regular semisimple element in $G = SL_2(\mathbb{C})$ and $y$ to be a regular unipotent element in the same group.
For instance, take
$$x = \begin{pmatrix}
\lambda & \\ & \lambda^{-1}
\end{pmatrix}, \lambda \ne 1, 
\hskip5mm
\mathrm{and}
\hskip5mm
y = \begin{pmatrix}
1 & 1 \\ & 1
\end{pmatrix}$$
then the centraliser of $x$ is isomorphic to $\mathbb{G}_m$ and the centraliser of $y$ is isomorphic to $\mathbb{G}_a$. 
Hence, $x$ and $y$ are $dz$-equivalent, $\dim(\mathbb{G}_m) = \dim(\mathbb{G}_a) = 1$, but not $iz$-equivalent as $\mathbb{G}_m$ and $\mathbb{G}_a$ are not isomorphic.
\end{example}

\begin{remark}
Note that in the Example \ref{iz-but-not-z}, the two centralisers are conjugate in $SL_3$, a bigger group than $U$. 

There are two ways to see this. 
The elements $x$ and $y$ are themselves conjugate in $SL_3$ and hence their centralisers are conjugate in $SL_3$ which, in this case, agree with their centralisers in $U$.
Alternatively, the two centralisers correspond to two different simple systems of roots in the root system $\Phi$ of $SL_3$, hence they are conjugate by a Weyl group element.

If we were to consider only abstract groups, then the celebrated HNN-theory tells us that two isomorphic centralisers in a group become conjugate in a bigger group. 
\end{remark}

Emboldened by the above remark and especially the Example \ref{iz-but-not-z}, we ask

\begin{question}
If $G$ is a linear algebraic group and $H_1$, $H_2$ are two isomorphic subgroups of $G$, then is there a linear algebraic group containing $G$ in which $H_1$ and $H_2$ become conjugate?

Equivalently, if $H_1, H_2 \subseteq G$ are isomorphic then do we have an embedding of $G$ in some $GL_n$ such that $H_1, H_2$ become conjugate in $GL_n$?
\end{question}

Such a question for $iz$-equivalence vis-\'{a}-vis $z$-equivalence would be difficult to consider because, unlike Example \ref{iz-but-not-z}, the centralisers may change in a bigger group.
We therefore only ask how the $iz$-equivalence in $G$ behaves with respect to the embeddings of $G$ in bigger groups. 
Do $iz$-equivalent elements in $G$ remain $iz$-equivalent in all such embeddings? 
Do they become $z$-equivalent in some embedding? 
What would be the situation for finite groups of Lie type? 

Such questions will not make sense for other equivalences. 
We may have a semisimple element $z$-equivalent to a unipotent element, in an abelian group for instance, but these two elements will never be conjugate in a bigger group.
Indeed, they may not remain $z$-equivalent in a bigger group. 

Consider, for instance, the two distinct elements $1 = x, y$ in $B_2(\mathbb{F}_2)$, the Borel subgroup in $GL_2(\mathbb{F}_2)$. 
Since the group $B_2(\mathbb{F}_2)$ is abelian, $x$ and $y$ are $z$-equivalent but they don't remain $z$-equivalent in bigger groups, in $B_2(\mathbb{F}_4)$ for instance. 

Similarly, $x$ and $y$ in Example \ref{dz-but-not-iz} can never be $iz$-equivalent in a bigger group.

\begin{remark}
The number of $iz$-classes in group $U$ above is finite. 
However, we do not know if this holds in general.
If we knew that the number of $iz$-classes in an algebraic group is finite, at least, for a class of groups (other than the reductive groups), then we would be able to generalise the proof of Theorem \ref{reductive} for this class of groups. 
However, we have no such information at present. 
We do not even know if there is a dense $iz$-class in a class of algebraic groups (other than the reductive groups). 

Hence we consider the $dz$-classes, whose number is finite for every algebraic group.
It also follows that every algebraic group contains a dense $dz$-class.
\end{remark}

\begin{definition}\label{def-regular}
Let $G$ be a linear algebraic group. 
An element $g \in G$ is called {\em regular} if the dimension of its centraliser, $Z_G(g)$, is minimum among such dimensions. 

Equivalently, an element $g \in G$ is regular if its conjugacy class has the largest possible dimension. 
\end{definition}

We note that the set of regular elements in a $dz$-class, which is open (and hence dense) in $G$ and the centralisers of two elements in this class have the same dimensions. 
These were the three properties that we had listed at the beginning of this section.

\begin{definition}
Let $G$ be a linear algebraic group.
The \emph{regular rank} of $G$ is the dimension of the centraliser of a regular element in $G$. 
\end{definition}

We are now ready to prove the general version of Theorem \ref{reductive}.

%%%%%%%%%%%%%%%%%%%%%%%%%%%%%%%%%%%%
\section{Non-reductive groups}
In this section, $G$ is a connected, linear algebraic group defined over $\mathbb{C}$. 
The set of regular elements in $G$ is denoted by $G_{reg}$. 

We note some basic results about the regular rank.
The proofs are omitted. 

\begin{remark}\label{regular}
\begin{enumerate}
\item If $G$ is reductive then the regular rank of $G$ is the same as its rank, the dimension of a maximal torus in $G$. 
\item If $G_1 \subseteq G_2$ then the regular rank of $G_1$ is less than or equal to that of $G_2$. 
\item The regular rank of $G_1 \times G_2$ is the sum of the regular ranks of $G_i$. 
\item The regular rank of a semidirect product $G_1 \ltimes G_2$ need not be equal to the sum of the regular ranks of $G_i$. 
We will see an example of this, a Borel in a semisimple group, in Lemma \ref{Borel}.
\end{enumerate}
\end{remark}

\begin{theorem}\label{general}
Let $G$ be a complex, connected, linear algebraic group. 
If the dimension of $G$ is $n$ and the regular rank of $G$ is $r$ then the commuting probability of $G$ is 
$$\frac{n + r}{2n} = \frac{1}{2} + \frac{r}{2n}.$$
\end{theorem}

\begin{proof}
The proof develops on the similar lines as the proof of Theorem \ref{reductive}. 

We note that $C(G)$ is equal to the union $\cup_{\mathfrak{c} \in \mathcal{C}} C(G)_{\mathfrak{c}}$ where $\mathcal{C}$ is the set of conjugacy classes in $G$ and $
\dim(C(G)_{\mathfrak{c}}) = \dim(G)$ for every $\mathfrak{c} \in \mathcal{C}$. 

Since the set $G_{reg}$ is dense in $G$ and is a union of conjugacy classes in $G$, we have that $\dim(C(G))$ is equal to $\dim(G_{reg}/\mathfrak{c}) + \dim(G)$ 
where the set $G_{reg}/\mathfrak{c}$ is the set of conjugacy classes in $G_{reg}$. 
The dimension of $G_{reg}/\mathfrak{c}$ is equal to $\dim(G) - \dim(\mathfrak{c}) = \dim Z_G(g)$, where $g$ is a fixed element in $\mathfrak{c}$.
Thus, $\dim(G_{reg}/\mathfrak{c})$ is the regular rank of $G$. 
Hence $\dim(C(G)) = n + r$ and 
$$p(G) = \frac{n + r}{2n} = \frac{1}{2} + \frac{r}{2n}.$$
\end{proof}

\begin{remark}
For a connected linear algebraic group $G$, if the set $\mathcal{C}$ of conjugacy classes forms a variety in a natural way then it follows from the above proof that $\dim(\mathcal{C})$ is the same as the regular rank of $G$. 
\end{remark}

We now compute the regular ranks of certain subgroups of semisimple groups. 

\begin{lemma}\label{Borel}
Let $G$ be a complex semisimple group. 
We fix a Borel subgroup $B$ in $G$ and let $U$ be the unipotent radical of $B$. 
Then the regular rank of $U$ is equal to the rank of the group $G$. 

The regular rank of a parabolic in $G$ is equal to the rank of $G$. 
In particular, the regular rank of $B$ is equal to the rank of $G$. 
\end{lemma}

\begin{proof}
Let $u$ be a regular unipotent element of $G$ which belongs to $U$.
It is proved by Steinberg, \cite[$\S 4$]{St}, that such elements exist. 
The dimension of $Z_G(u)$ is equal to the rank of $G$ which is bigger than or equal to the regular rank of $U$. 
Further, since $u$ is in a unique Borel subgroup of $G$, it follows that $Z_G(u) = Z_U(u)$. 
Hence the regular rank of $U$ is equal to the rank of $G$. 

If $P$ is a parabolic in $G$ then, up to conjugacy in $G$, $U \subset P \subset G$. 
It follows, from Remark \ref{regular} (2), that the regular rank of $P$ is also equal to the rank of $G$. 
\end{proof}

\begin{corollary}
Let $G$ be as in the above lemma with dimension $n$ and rank $r$, and let $U$ be as in the above lemma. 
Then 
$$p(U) = \frac{(\frac{n-r}{2}) + r}{n - r} = \frac{1}{2} + \frac{r}{n-r} .$$
\end{corollary}

\begin{proof}
Follows from the above two results and that $2 \dim(U) + r = n$.  
\end{proof}

\begin{remark}\label{unipotent}
Using the above result, we note commuting probabilities of the unipotent radicals of Borel subgroups of all simple groups.
For a simple $G$, the corresponding unipotent radical is denoted by $U(G)$. 
\begin{enumerate}
\item $p(U(SL_n)) = \dfrac{1}{2} + \dfrac{n-1}{n^2 - n} = \dfrac{1}{2} + \dfrac{1}{n + 1}$, for $n \geq 1$.
\vskip1mm
\item $p(U(SO_{2n+1})) = \dfrac{1}{2} + \dfrac{n}{2n^2} = \dfrac{1}{2} + \dfrac{1}{2n}$, for $n \geq 2$.
\vskip1mm
\item $p(U(Sp_{2n})) = \dfrac{1}{2} + \dfrac{n}{2n^2} = \dfrac{1}{2} + \dfrac{1}{2n}$, for $n \geq 3$.
\vskip1mm
\item $p(U(SO_{2n})) = \dfrac{1}{2} + \dfrac{n}{2n^2 - 2n} = \dfrac{1}{2} + \dfrac{1}{2n - 2}$, for $n \geq 4$.
\vskip1mm
\item $p(U(G_2)) = \dfrac{2}{3}$, \hskip2mm $p(U(F_4)) = p(U(E_6)) = \dfrac{7}{12}$, \hskip2mm $p(U(E_7)) = \dfrac{5}{9}$ and \hskip2mm \noindent $p(U(E_8)) = \dfrac{27}{50}$.
\end{enumerate}
\end{remark}

%%%%%%%%%%%%%%%%%%%%%%%%%%%%%%%%%%%%
\section{Limit points}

It follows from Theorem \ref{general} that $\frac{1}{2} < p(G) \leq 1$. 
We also have that $p(G) = 1$ if and only if $G$ is abelian. 
In this section, we investigate the possible values of $p(G)$ in $[\frac{1}{2}, 1]$ as $G$ varies over all linear algebraic groups. 
We will also compute the limit points of these numbers. 

\begin{lemma}\label{simplebounded}
\begin{enumerate}
\item An $\alpha \in [\frac{1}{2}, 1]$ is $p(G)$ for a simple $G$ if and only if $\alpha = \frac{1}{2} + \frac{1}{2m}$ for some $m > 1$. 

\item The number of simple groups $G$ with $p(G) > p/q > 1/2$ is finite. 

\item The set of these numbers, from $(1)$ above, has only one limit point which is $\frac{1}{2}$.
\end{enumerate}
\end{lemma}

\begin{proof}
The first part of the lemma is evident from Remark \ref{simple}. 

The inequality $p(G) > p/q$ gives $1/2 + r/2n > p/q$ which gives $r/n > (2p-q)/q > 0$ and $n/r < q/(2p-q)$.
The number $n/r$ is an integer for a simple algebraic group and it follows, from Remark \ref{simple} for instance, that there are only finitely many simple $G$ whose $n/r$ is bounded above, for every bound. 
Thus, there are only finitely many simple $G$ with $p(G) > p/q > 1/2$. 

Finally, from (2) it follows that $1/2$ is the only possible limit point of the set of numbers $p(G)$ where $G$ varies over simple algebraic groups. 
We see that $1/2$ is indeed a limit point, $\lim_{n \to \infty}p(SL_n) = 1/2$. 
\end{proof}

If we consider the reductive groups, instead of only the simple ones, then the above lemma does not hold. 
For example, if $G = GL_2 \times \mathbb{G}_m^a$ then $p(G) = \frac{3 + a}{4 + a}$ and the limit of these numbers is $1$!
In fact, we have the following result:

\begin{proposition}
Every rational number from the set $(\frac{1}{2}, 1]$ is $p(G)$ for some reductive group $G$. 
\end{proposition}

\begin{proof}
Assume that we have a rational number $1/2 < p/q \leq 1$. 
Choose an $n$ with $1/2 < (n+1)/2n \leq p/q$. 
This is possible as $\{(n+1)/2n\}$ is a decreasing sequence whose limit is $1/2$.

Let $G = GL_n^a \times \mathbb{G}_m^b$ where $a = q - p$ and $b = (2n^2p-n^2q - nq)/2$. 

Here, $a \geq 0$ as $p/q \leq 1$. 
Further, $b$ is an integer as $n^2 - n$ is always even and $b \geq 0$ if and only if $2np - (n+1) q \geq 0$ which holds as $(n+1)/2n \leq p/q$. 

The rank of $G$ is equal to $an + b$ and the dimension is equal to $an^2 + b$. 
Then
$$p(G) = \frac{(an^2 + b) + (an + b)}{2(an^2 + b)} = \frac{(n^2-n)p}{(n^2-n)q} = \frac{p}{q} .$$
\end{proof}

We have the same result for nilpotent groups. 

\begin{proposition}
Every rational number from the set $(\frac{1}{2}, 1]$ is $p(G)$ for some nilpotent group $G$. 
\end{proposition}

\begin{proof}
The proof follows exactly in the similar way as the proof of the above proposition. 
The role of $\mathbb{G}_m$ in the above proof will be played by $\mathbb{G}_a$ here. 
\end{proof}

\begin{corollary}\label{limit}
\begin{enumerate}
\item $\{p(G): G$ is reductive$\} = \{p(G): G$ is nilpotent$\} = \{p(G): G$ is solvable$\} =  \{p(G): G$ is a linear algebraic group$\} = \mathbb{Q} \cap (1/2, 1]$. 
\item Every $\alpha \in [1/2, 1]$ is a limit point of each of the sets in $(1)$ above. 
\end{enumerate}
\end{corollary}

%%%%%%%%%%%%%%%%%%%%%%%%%%%%%%%%%%%%
\section{Comparisons with the finite case}

As mentioned in the introduction, the notion of $p(G)$ for finite groups has been studied extensively.
A review of whether the analogous results hold for our notion of $p(G)$ is in order.

\subsection{Simple groups}
If $G$ is a non-abelian finite simple group then $p(G) \leq 1/12$, as was observed by J. Dixon (\cite[Introduction]{GR}). 
The probabilities $p(G)$ for simple algebraic groups are also well-behaved, Lemma \ref{simplebounded}.

\subsection{Direct products}
If $G_1$ and $G_2$ are finite groups then $p(G_1 \times G_2) = p(G_1) p(G_2)$. 
However, in the case of algebraic groups $p(G_1 \times G_2)$ is almost never equal to $p(G_1)p(G_2)$. 
One way to understand this is that $p(G)$ for an algebraic group is always bounded below by $1/2$. 
The multiplicativity of $p(G)$ for direct products of algebraic groups would say that for a non-abelian $G$ and a sufficiently large integer $r$, $p(G^r) = p(G)^r < 1/2$ giving a contradiction. 

\subsection{Solvable groups and non-solvable groups}
If $G$ is a finite group with $p(G) > 3/40$ then Guralnick and Robinson prove that $G$ is either solvable or is isomorphic to $A_5 \times T$ for an abelian group $T$ (\cite[Theorem 11]{GR}). 
Thus, the numbers $p(G)$ for non-solvable groups are bounded above by $3/40$, except for $p(A_5) = 1/12$. 
In contrast, the numbers $p(G)$ for non-abelian reductive (hence non-solvable) algebraic groups $G$ take all rational values in $(1/2, 1)$, Corollary \ref{limit} (1). 

\subsection{Limit points} 
If $G$ is a finite non-abelian group then $p(G) \leq 5/8$ (\cite[Introduction]{G}).
Further, there are gaps within the numbers $p(G)$. 
However, the numbers $p(G)$ for algebraic groups take all rational values in $(1/2, 1]$, Corollary \ref{limit} (1). 
As a result, every real number in $[1/2, 1]$ is a limit point of the set $\{p(G): G$ is algebraic$\}$ whereas the set of limit points in the finite case is a nowhere dense set of rational numbers (\cite[Corollary 1.3]{E}).

\subsection{Isoclinism}
Two groups $G$ and $H$ are called isoclic if there are isomorphisms $\phi:G/Z(G) \to H/Z(H)$ and $\psi:G' \to H'$ such that the diagram
$$\begin{tikzcd}
G/Z(G) \times G/Z(G) \arrow{r}{\phi \times \phi} \arrow[swap]{d}{\alpha_G} & H/Z(H) \times H/Z(H) \arrow{d}{\alpha_h} \\
G' \arrow{r}{\psi} & H'
\end{tikzcd}$$
is commutative, where $\alpha_G(aZ(G), bZ(G)) = aba^{-1}b^{-1}$ and $\alpha_H$ is defined analogously. 

It is proved in \cite[Lemma 2.4]{L} that if $G$ and $H$ are finite isoclinic groups then $p(G) = p(H)$. 

We can define isoclinism in exactly the same way as above for algebraic groups. 
Then $GL_n$ and $SL_n$ are isoclinic groups with different commuting probabilities. 

%%%%%%%%%%%%%%%%%%%%%%%%%%%%%%%%%%%%
\section{Compatibility with the finite groups of Lie type}

It seems from the above section that there are more differences than similarities when we compare $p(G)$ for algebraic groups with the $p(G)$ for finite groups. 
This is not surprising as our definition of $p(G)$ for algebraic groups uses the dimension, which is additive for direct products, and then we take the ratios of the dimensions. 
This is the central reason for many differences noted in the previous section. 
In spite of this, we contend that our notion of $p(G)$ is a natural notion by comparing it with $p(G)$ of the finite groups of Lie type. 

Clearly, the characteristic of the field $\mathbb{C}$ played no role in the proofs of our results. 
Hence the results stand good for groups defined over an algebraically closed field. 
In fact, if $G$ is a connected, linear algebraic group defined over a perfect field $k$ then the dimension of $G$ and the regular rank of $G$ is defined over $k$. 
Hence the results proved in the paper hold good over the field $k$ as well. 

Let $k = \mathbb{F}_q$ be the finite field with $q$ elements and let $G$ be a semisimple group defined over $k$ of dimension $n$ and rank $r$. 
We recall a result of Steinberg. 

\begin{theorem}[Steinberg]\cite[14.11]{St1}
The number of semisimple conjugacy classes in $G(k)$ is equal to $q^r$.
\end{theorem}

Further, the number of all conjugacy classes in $G(\mathbb{F}_q)$ is $O(q^r)$ and the cardinality of $G(\mathbb{F}_q)$ is $O(q^n)$. 
The commuting probability of the finite group $G(\mathbb{F}_q)$ is therefore $O(q^r)/O(q^n) = O(q^{n+r})/O(q^{2n})$ which is compatible with the commuting probability of the algebraic group $G$. 

%%%%%%%%%%%%%%%%%%%%%%%%%%%%%%%%%%%%
\section{Concluding remarks}

\subsection{Non-affine groups}
The set $C(G)$ makes sense even when $G$ is not an affine algebraic group and hence we can define $p(G)$ in exactly the same way for a non-affine $G$ as in the affine case. 
Further, the notion of a regular element also makes sense for an algebraic group $G$, the set of regular elements is a dense subset of $G$ and hence our main result, Theorem \ref{general}, remains valid that $p(G) = (n+r)/2n$. 
However, at this point, we do not have any computation of $p(G)$ for a non-affine algebraic group, except when $G$ is an abelian variety with $p(G) = 1$. 

\subsection{Further possibilities}
While the set $C(G)$ of commuting pairs in a reductive group $G$ has been studied previously, no one seems to have introduced the notion of the commuting probability of $G$.
The notion in the case of finite groups has many applications and has been extensively investigated. 
We hope that this notion also gets investigated. 

There are many ways in which the probabilistic studies of algebraic groups can be done. 
We list some possibilities below.
It is our hope that they will also be taken up. 

\begin{question}
We prove in Section 5 that every rational number in $(1/2, 1]$ is $p(G)$ for a linear algebraic group $G$.
However, in the proof we have used groups that have a large abelian direct factor. 
What are the value sets of $p(G)$ where $G$ has no abelian direct factor? 
What are the limit points of this set? 
\end{question}

\begin{question}
If $H$ is an algebraic subgroup of $G$ then we can define $p(G, H)$ as 
$$p(G, H) = \frac{\dim(C(G, H))}{\dim(G \times H)} $$
where $C(G, H) = \{(g, h) \in G \times H: gh = hg\}$. 
This $p(G, H)$ gives the probability that two randomly chosen elements $g \in G$ and $h\in H$ commute with each other. 

Following the sections 2--4, it does not seem difficult to compute $p(G, H)$ when $H$ is a parabolic in a reductive $G$. 
It would be interesting to study this notion in general, especially when $G$ and $H$ are unipotent. 
We have no information on the difficulty level of this question in the unipotent case. 
\end{question}

\end{document}